\documentclass{amsart}

\usepackage{fontenc}
\usepackage{amssymb,amsmath,amscd,latexsym,amsfonts,amstext,amsbsy,amsthm}
\usepackage{euscript}
\usepackage{enumerate}

\usepackage[centertags]{amsmath}
\usepackage{amsfonts}
\usepackage{amssymb}
\usepackage{amsthm}
\usepackage{newlfont}

 \newtheorem{theorem}{Theorem}[section]

 \newtheorem{proposition}[theorem]{Proposition}
 
 \theoremstyle{definition}

 \numberwithin{equation}{subsection}
  
 \newcommand{\norm}[1]{\left\Vert#1\right\Vert}

\newcommand{\bp}{\begin{problem}}

\newcommand{\ep}{\end{problem}}

\newcommand{\ben}{\begin{enumerate}}
\newcommand{\een}{\end{enumerate}}

\begin{document}

\title{Borel Spectrum of Operators on  Banach Spaces}
\author{Mohammed Yahdi}

\address{Department of  Mathematics and
Computer Science\\
Ursinus College\\
Collegeville,  PA 19426, USA}

\email{myahdi@ursinus.edu}

\subjclass[2000]{Primary 54C10, 47A10, 54H05.}

\keywords{Borel Function, Banach Space,  Operator, Polish Space,
Spectrum.}

\maketitle


\begin{abstract}
\noindent \emph{The paper investigates  the variation of the spectrum
 of operators in infinite dimensional Banach spaces. In particular,
it is shown that the spectrum function is Borel from the space of bounded
operators on a separable Banach space; equipped with
the strong operator topology, into the Polish space
 of compact subsets of the closed unit disc
of the  complex plane; equipped with the
Hausdorff topology.  Remarks and results are given  when other
topologies are used.}

\end{abstract}


\section{Preliminary}

Let $X$ be an infinite dimensional Banach space. We denote by $T$ an
arbitrary bounded operator on $X$ and by $I$  the identity operator
on $X$. Let $\mathbb{D}$  be the closed unit disc of the complex
plane $\mathbb{C}$. The restriction on $\mathbb{D}$ of the spectrum of an operator $T$, denoted by  $ \sigma(T)$,
is defined as follows:
\[
\mathcal{\sigma}(T) := \{ \lambda \in \mathbb{D};
 ~ (T- \lambda \mathnormal{I}) \hbox{ is not invertible } \}.
 \]
Recently, essential spectra of some matrix
operators on Banach spaces (see \cite{JERIBI}) and spectra of some block operator matrices (see \cite{Jeribi2})
were investigated, with applications ro  differential and transport operators
This paper investigates the variations of the spectrum
$\mathcal{\sigma}(T)$ as $T$ varies over the space $L(X)$ of  all bounded operator on the Banach space $X$.
 First, we introduce the sets
and the topologies required for this study. We denote by
\begin{itemize}
  \item $\mathcal{K}(\mathbb{D})$ the set of all compact subsets of
  the closed unit disc $\mathbb{D}$ of the complex plane
  $\mathbb{C}$,
  \item $\mathcal{\sigma}$ the spectrum function defined from $L(X)$
  into $\mathcal{K}(\mathbb{D})$ that maps an operator $T$ to its
  spectrum $\mathcal{\sigma}(T)$.
\end{itemize}

The set $\mathcal{K}(\mathbb{D})$  is endowed with the Hausdorff
topology generated by the families  of all subsets in one of the
following forms
\[\big{\{}F \in \mathcal{K}(\mathbb{D}); ~ F \cap V \ne \emptyset\big{\}}
\quad \text{ and } \quad
 \big{\{}F \in \mathcal{K}(\mathbb{D}); ~ F \subseteq V  \big{\}},\]
for $V$  an open subset of  $\mathbb{D}$. Therefore,
$\mathcal{K}(\mathbb{D})$ is a Polish space, i.e., a separable
metrizable complete space, since $\mathbb{D}$ is Polish (see
\cite{KL},\cite{Ku} or \cite{C}). It is shown below that we can
reduce the families that generate the above Hausdorff topology.

\begin{proposition} \label{prop:hausdorff}
 Let $\mathcal{K}(\mathbb{D})$ be the set of compact subsets of the closed unit disc $\mathbb{D}$.
 Then  $\mathcal{K}(\mathbb{D})$  equipped with the Hausdorff topology is a
Polish space;  where the Borel structure is generated by one
of the following two  families

\[
\Big{\{}\{K \in \mathcal{K}(\mathbb{D}): K \cap V \neq \emptyset\};
 V \hbox{ open in } \mathbb{D}  \Big{\}}
\]
\[
\Big{\{}\{K \in \mathcal{K}(\mathbb{D}): K \subset V \};
 V \hbox{ open in } \mathbb{D}  \Big{\}}.
\]
\end{proposition}

\begin{proof}

Let  $V$ be an open subset of  $D$. There exists a decreasing
sequence of open subsets $(O_n)_{n\in \mathbb{N}}$ such that
 $V^c=\bigcap_{n\in \mathbb{N}}O_n$;
for example $O_n=\{x \in \mathbb{D}: dist(x,V^c) \leq
\frac{1}{n}\}$.
We have

\begin{align*}
\{K \in \mathcal{K}(\mathbb{D}): K \cap V \neq \emptyset\}^c
=&\{K \in \mathcal{K}(\mathbb{D}): K \subseteq V^c \}\\
=&\bigcap_{n\in \mathbb{N}} \{K \in \mathcal{K}(\mathbb{D}): K
\subseteq O_n \}.
\end{align*}

On the other hand,
\begin{align*}
\{K \in \mathcal{K}(\mathbb{D}): K \subseteq V \}^c
=&\{K \in \mathcal{K}(\mathbb{D}): K \cap V^c \neq \emptyset\}\\
=&\{K \in \mathcal{K}(\mathbb{D}):  K \cap O_n \neq \emptyset, ~ \forall n \in \mathbb{N}\}\\
=&\bigcap_{n\in \mathbb{N}}\{K \in \mathcal{K}(\mathbb{D}):  K \cap
O_n \neq \emptyset\}.
\end{align*}
%
Indeed, if for all $n \in\mathbb{N}$, there exists $x_n \in K \cap O_n$, then
there exists a subsequence $(x_{n_k})_k$ of $(x_{n})_n$ that
converges to $x \in K$, and $x \in \bigcap_n O_n$ since $(O_n)_n$ is
decreasing.

\end{proof}

\section{Norm Operator Topology and  the Spectrum Function}

 We equip $L(X)$ with
 the canonical  norm of operators defined by $\norm{T}= \sup_{x \in B_X}
 \norm{T(x)}$, where $B_X$ is the unit ball of $X$.
Note that the map $\mathcal{\sigma}: T \longmapsto   \sigma(T)$ is
 not continuous when $L(X)$ is endowed
with  its canonical  norm. Indeed, the operators
$T_n=(1+\frac{1}{n})I$ converge to the identity $I$ while
$\mathcal{\sigma}(T_n)=\emptyset$ and $\mathcal{\sigma}(I)=\{1\}$.
However, we have the following result.
\begin{proposition}
Let $X$ be a Banach space, $(L(X), \norm{.})$ the space of bounded
operators equipped with the norm of operators, and
$\mathcal{K}(\mathbb{D})$ the set of compact subsets of the unit
disc $\mathbb{D}$ equipped with the Hausdorff topology. Then the
spectrum map
\begin{eqnarray*}
\mathcal{\sigma}:  \big{(}L(X),\norm{.}\big{)} & \longrightarrow &  \mathcal{K}(\mathbb{D}) \\
                    T & \longmapsto &   \sigma(T)
\end{eqnarray*}
is upper-semi continuous.
\end{proposition}

\begin{proof}

 Let  $V$ be an open subset  $\mathbb{D}$. By proposition \ref{prop:hausdorff}, it is only  need to show that
 the set $O_V=\{T \in L(X); \sigma(T) \subseteq V\}$ is $\norm{.}$-open in $L(X)$.
 Let $T_0$ be fixed in $O_V$. Since
$\sigma(T_0)\cap \mathbb{D} \subseteq V$, then for all
 $\lambda \in \mathbb{D} \setminus V$, the operator $(T_0-\lambda I)$
is invertible and  the map   $\lambda \in \mathbb{D}
\setminus V \mapsto (T_0-\lambda I)^{-1}$ is continuous (see
\cite{schartz
}).
 It follows that

\[
\sup_{\lambda \in  \mathbb{D} \setminus V} \norm{(T-\lambda
I)^{-1}}< +\infty
\]
since $ \mathbb{D} \setminus V$ is compact.   Put

\[
\delta = \inf_{\lambda \in  \mathbb{D} \setminus V}
\frac{1}{\norm{(T_0-\lambda I)^{-1}}} >0.
\]

Let $T \in L(X)$ such that $\norm{T-T_0}<\delta$. For any $\lambda
\in \mathbb{D} \setminus V$ we have

\[
\norm{(T-\lambda I)-(T_0-\lambda
I)}=\norm{T-T_0}<\frac{1}{\norm{(T_0-\lambda I)^{-1}}}\cdot
\]
Thus, $(T-\lambda I)$ is invertible
 and hence
$\lambda \notin \sigma(T)$. In other terms, $\sigma(T) \subseteq V$
for all $T\in L(X)$ with $\norm{T-T_0}<\delta$. Therefore $O_V$ is
an open subset of $\big{(}L(X),\norm{.}\big{)}$.
\end{proof}

\section{Strong Operator Topology and  the Spectrum Function}

Consider now $L(X)$ equipped with  the strong operator topology
$S_{op}$ ( see \cite{schartz}).
In general,  $L(X)$ equipped with the strong operator topology is
not a polish space (since it is not a Baire space). However, if $X$
is separable, then
 $(L(x), S_{op})$ is  a standard Borel space. Indeed, it is Borel-isomorph to a Borel  subset of the Polish space
$ X^\mathbb{N}$ equipped with the norm product topology via the map
\begin{eqnarray*}
\varphi: \big{(}L(X), S_{op}\big{)}  & \longrightarrow &
\big{(} X^\mathbb{N}, \mathcal{P}\big{)} \\
T   & \longmapsto &  (Tz_n)_{n \in \mathbb{N}},
\end{eqnarray*}
where $\{z_n, n \in \mathbb{N}\}$ is a dense $\mathbb{Q}$-vector
space in $X$.

Let us check how this topology on $L(x)$ affects the spectrum
function.

\begin{theorem} \label{prop:Psi}
For any separable infinite dimensional  Banach  $X$, the map
\begin{eqnarray*}
\mathcal{\sigma}:  L(X) & \longrightarrow &  \mathcal{K}(\mathbb{D}) \\
                    T & \longmapsto &   \sigma(T),
\end{eqnarray*}
 which maps a bounded operator  to its
spectrum, is Borel when $L(X)$ is endowed with the strong operator
topology $S_{op}$ and $\mathcal{K}(\mathbb{D})$ with the Hausddorf
topology.
 \end{theorem}

\begin{proof}
As $\mathcal{K}(\mathbb{D})$ is equipped with the  Hausdorff
topology, it follows from the proposition  \ref{prop:hausdorff},
that it is enough to show that for any open subset  $V$ of the disc
$\mathbb{D}$, the subset
\[
E_V= \{T \in L(X): \sigma(T) \cap V \neq \emptyset\}
\]
is Borel  in $\big{(}L(X),S_{op}\big{)}$.

Let $V$ be a fixed open subset of  $\mathbb{D}$. We have
\begin{align*}
E_V 
    &= P_{L(X)}(\Omega),
\end{align*}
where $P_{L(X)}$  stands for the canonical projection  of  $L(X)
\times \mathbb{D}$  onto $L(X)$, and
\[
\Omega = \big{\{}(T,\lambda) \in L(X) \times V: ~ \lambda \in
\sigma(T) \big{\}}.
\]

By a  descriptive set theory result from \cite{SR}, to show that
$E_V$ is a Borel set it suffices  to show that $\Omega$ is a Borel
set with $K_\sigma$ vertical sections.

For $T \in L(X)$, the
vertical section of the set $\Omega \subseteq L(X) \times
\mathbb{D}$ along the direction
 $T$ is given by
\begin{align*}
\Omega(T) &= \big{\{} \lambda \in \mathbb{D}: ~ (T,\lambda) \in  \Omega \big{\}}\\
          &= \big{\{} \lambda \in \mathbb{D}: ~ \lambda \in  V \cap \sigma(T) \big{\}}\\
          &= \sigma(T) \cap V.
\end{align*}
Thus, it is a $K_\sigma$ of $\mathbb{D}$.

Now, we need to prove that $\Omega$ is a Borel set. Put
\[ \Delta =
\big{\{}(T,\lambda) \in L(X) \times \mathbb{D}: ~ \lambda \in
\sigma(T) \big{\}}.
\]
Therefore
\[
\Omega = \Delta ~ \cap ~ L(X) \times V,
\]
Hence, to finish the proof, it is enough to prove the following claim.
\vskip0.8cm

\underline{\emph{Claim:}} $\Delta$ is a Borel set of $ L(X) \times
\mathbb{D}$.

First, note that   \quad $ \Delta = A \cup B$ \quad with
\begin{itemize}
  \item $A= \big{\{}(T,\lambda) \in L(X) \times \mathbb{D}: T-\lambda I
            \hbox{ is not   isomorph to its range }\big{\}}$
  \item $B= \big{\{}(T,\lambda) \in L(X) \times \mathbb{D}: ~ (T-\lambda
I)(X)
 \hbox{ is not dense in } X \big{\}}.$
\end{itemize}
Indeed, if  ~ $ T-\lambda I $ ~ is an isomorphism onto its range,
then $~ (T-\lambda I)(X) ~ $ is a closed subspace that will be
strict if $\lambda \in \sigma(T)$, and thus not dense in $X$.

On the other hand, since $X$ is separable, there exists a countable and dense subset $
\mathcal{Y} $ in the  sphere $S_X$ of $X$, and there exists a dense
sequence $\{x_n\}_{n\in \mathbb{N}} $ in $X$.

Now,
we will show that $A$ and $B$ are Borel sets.
Let $(T,\lambda) \in
L(X) \times \mathbb{D} $. From the definition of  $A$, We have
 $(T,\lambda) \in A $ if and only if
\[
\exists (z_n)_{n\in \mathbb{N}} \subseteq S_X : \quad \lim_{n \to
\infty} \norm{(T-\lambda I)z_n}=0 .
\]
In other term, this is equivalent to
\[
\exists (z_n)_{n\in \mathbb{N}} \subseteq \mathcal{Y}, ~ \forall k
\geq 1 ~ \exists N_k \in \mathbb{N} ~ \forall n \geq N_k: \quad
\norm{(T-\lambda I)z_n} < \frac{1}{k}.
\]
By choosing the subsequence $(z_{N_k})_{k\in \mathbb{N}}$ instead of
$(z_{n})_{n\in \mathbb{N}}$, the previous statement is equivalent to
\[
\exists (z_n)_{n\in \mathbb{N}} \subseteq \mathcal{Y}, ~ \forall k
\geq 1 ~ \exists N_k \in \mathbb{N} ~ : \quad \norm{(T-\lambda
I)z_{N_k}} < \frac{1}{k},
\]
or again,
\[
\forall k \geq 1, ~ \exists x \in \mathcal{Y}: \quad
\norm{Tx-\lambda x} < \frac{1}{k}.
\]
Therefore,
\[
A= \bigcap_{k\geq 1} \bigcup_{x \in \mathcal{Y}} A^x_k
\]
with
\[
A^x_k= \big{\{}(T,\lambda) \in L(X) \times \mathbb{D}: ~
\norm{Tx-\lambda x} < \frac{1}{k}
             \big{\}}.
\]
Since $L(X)$ is equipped with the the strong operator convergence
$S_{op}$, it follows that
 $A^x_k$ are open sets.
Hence, $A$ is a Borel set.

On the other hand, ``$(T-\lambda I)(X)$ is not dense in $X$'' is
equivalent to
\[
  \exists y \in
  S_X \text{ and }
             \exists k \geq 1 \hbox{ such that  } \forall x \in X:
             \norm {y-(T-\lambda I)x} \geq \frac{1}{k},
             \]
or again,
  \[ \exists y \in \mathcal{Y} ~ \hbox{ and }
             \exists k \geq 1 \hbox{ such that } ~ \forall n \in \mathbb{N}: ~
             \norm {y-(T-\lambda I)x_n} \geq \frac{1}{k}.
  \]
  Therefore
  \[B=\bigcup_{y \in \mathcal{Y}} \bigcup_{k \in \mathbb{N}} \bigcap_{n \in \mathbb{N}}
  B^y_{k,n} \]
with
\[
B^y_{k,n}=\Big{\{}(T,\lambda) \in L(X) \times \mathbb{D}; \quad
                   \norm {y-(T-\lambda I)x_n} \geq \frac{1}{k}  \Big{\}}.
\]
Similarly to $A^x_k$, it is not difficult to see that the sets
$B^y_{k,n}$ are Borel sets.
Hence $B$ is also a Borel set. This proves the claim and ends the
proof of the theorem  \ref{prop:Psi}.

\end{proof}


\end{document}